\documentclass[11pt,reqno]{amsart}
\usepackage{graphicx}
\usepackage{float}
\usepackage[caption = false]{subfig}
\usepackage{amsmath}
\usepackage{enumerate}
\usepackage{mathtools}
\usepackage[english]{babel}
\usepackage{amsfonts,amssymb}
\usepackage{mathrsfs}
\usepackage[utf8x]{inputenc}\usepackage[english]{babel}
\usepackage{soul}
\usepackage[all]{xy,xypic}
\usepackage{cite}
\usepackage{relsize}
\numberwithin{equation}{section}
\newtheorem{theorem}{Theorem}[section]
\newtheorem{corollary}[theorem]{Corollary}

\theoremstyle{definition}
\newtheorem{definition}[theorem]{Definition}
\theoremstyle{remark}
\newtheorem{remark}[theorem]{Remark}
\newtheorem{example}{Example}

\numberwithin{equation}{section}
\DeclareMathOperator{\RE}{Re}

\begin{document}
	
	\title[\tiny{A Novel Class of Starlike Functions}]{A Novel Class of Starlike Functions}

	\author[S. Sivaprasad Kumar]{S. Sivaprasad Kumar}
	\address{Department of Applied Mathematics, Delhi Technological University, Delhi--110042, India}
	\email{spkumar@dce.ac.in}

	\author[S. Banga]{Shagun Banga}
	\address{Department of Applied Mathematics, Delhi Technological University, Delhi--110042, India}
	\email{shagun05banga@gmail.com}

	
	\maketitle
		\begin{abstract}
		In the past several subclasses of starlike functions are defined involving real part  and modulus of certain expressions of functions under study, combined by way of an inequality. In the similar fashion, we introduce a new class $\mathcal{S}^*_{q}(\alpha)$, consisting of normalized analytic univalent functions $f$ in the open unit disk $\mathbb{D}$,  satisfying $$\RE\left(\dfrac{z f'(z)}{f(z)}\right) > \left|1+\dfrac{z f''(z)}{f'(z)} -\dfrac{z f'(z)}{f(z)}-\alpha\right| \quad (0 \leq \alpha <1).$$ Evidently, $\mathcal{S}^*_{q}(\alpha) \subset \mathcal{S}^*$, the class of starlike functions. We first establish $\mathcal{S}^*_{q}(\alpha) \subset \mathcal{S}^*(q_\alpha)$, the class of analytic functions $f$ satisfying $z f'(z)/f(z)\prec q_\alpha(z),$ where $q_\alpha$ is an extremal function. We obtain certain inclusion and radius results for both the classes. Further, we estimate logarithmic coefficients, inverse coefficients and Fekete-Szeg\"o functional bounds for functions in $ \mathcal{S}^*(q_\alpha)$. 
	\end{abstract}
	\maketitle
	
	\section{Introduction}
	\label{intro}
	Let $\mathcal{A}$ be the class of analytic functions $f(z)= z+a_2 z^2 +a_3 z^3+\cdots,$ defined in a unit disk $\mathbb{D}=\{z\in \mathbb{C}: |z| <1\}$ and the subclass $\mathcal{S}$ of $\mathcal{A}$ consists of univalent functions. A function $f$ is said to be starlike if $k \mu \in f(\mathbb{D})$ for $\mu \in f(\mathbb{D})$ and $k \in [0,1]$. The class of starlike functions of order $\alpha$ $(0 \leq \alpha < 1)$ is defined by
	\begin{equation}\label{salpha} \mathcal{S}^*(\alpha):=\left\{f \in \mathcal{S}: \RE \dfrac{z f'(z)}{f(z)} > \alpha\right\}.\end{equation}
	For $\alpha=0$, $\mathcal{S}^*(0)=:\mathcal{S}^*$, the class of starlike functions. A function $f$ is said to be convex if $k \mu_1+(1-k) \mu_2 \in f(\mathbb{D})$ for $\mu_1$, $\mu_2 \in f(\mathbb{D})$ and $k \in [0,1]$. The class of convex functions of order $\alpha$ $(0 \leq \alpha <1)$ is defined by $$ \mathcal{K}(\alpha):=\left\{f \in \mathcal{S}: 1+\dfrac{z f''(z)}{f'(z)} > \alpha\right\}.$$ Note that $\mathcal{K}(0)=:\mathcal{K}$, the class of convex functions. Robertson~\cite{robert} introduced the classes $\mathcal{S}^*(\alpha)$ and $\mathcal{K}(\alpha)$ and functions in either of them fail to be univalent for $\alpha<0$. Moreover, there is a two way bridge  between these classes: $f \in \mathcal{K}(\alpha)$ if and only if $zf'(z) \in \mathcal{S}^*(\alpha)$. Another interesting class $\mathcal{M}(\beta) \subset \mathcal{A}$, studied by Uralegaddi~\cite{urla} is characterized by $\RE(z f'(z)/f(z)) < \beta$ $(\beta > 1)$. Carath\'{e}odory class, denoted by $\mathcal{P}$, consists of the functions $p(z)=1+p_1 z+p_2 z^2+\cdots,$ having positive real part in $\mathbb{D}$. Let $\mathcal{Q}(\alpha):=\{f \in \mathcal{A}: \RE\left(f(z)/z\right) > \alpha\}$ $(0 \leq \alpha <1).$ 	Recall $\mathcal{S}^*(1/2) \subset \mathcal{Q}(1/2)$ from~\cite[p.~73]{duren} and the constant $1/2$ is the best possible. Ma and Minda~\cite{ma} introduced the classes 
	$$\mathcal{S}^*(\phi)=\left\{f \in \mathcal{S}: \dfrac{ z f'(z)}{f(z)} \prec \phi(z)\right\} \text{ and } \mathcal{K}(\phi)=\left\{f \in \mathcal{S}: 1+\dfrac{z f''(z)}{f'(z)} \prec \phi(z)\right\},$$ where $\phi$ is analytic univalent function, which has positive real part in $\mathbb{D}$, $\phi(\mathbb{D})$ is symmetric with respect to real axis, $\phi'(0)>0$ and is starlike with respect to $\phi(0)=1$. The class $\mathcal{S}^*(\phi)$ reduces to various known subclasses of starlike functions on having particular functions in place of $\phi(z)$. For example $\phi(z)= (1+(1-2\alpha)z)/(1-z)$, it becomes starlike class of order $\alpha$, given in~\eqref{salpha}. Let $\phi(z)= \sqrt{1+z}$, then it reduces to $\mathcal{S}^*(\sqrt{1+z})$, introduced in~\cite{Sok}. The class $\mathcal{S}^*_{e}$, introduced by Mendiratta et~al.~\cite{men} when $\phi(z)= e^z$. Goel and Kumar~\cite{priyanka} introduced the class $\mathcal{S}^*_{SG}$ for $\phi(z) = 2/(1+e^{-z})$. The class $\mathcal{S}^*_{S}$ is obtained by taking $\phi(z)= 1+\sin(z)$, which is introduced in~\cite{cho}. Let $\phi(z)= 1+4z/3+2z^2/3$, then we get $\mathcal{S}^*_{C}$, first studied in~\cite{sharma}. Kumar and Banga~\cite{banga} introduced and studied the class $\mathcal{S}^*_{l}$, which is obtained by considering a special type of Ma-Minda function $\Phi(z)=1-\log(1+z)$.
	
	If there exist a Schwarz function $\upsilon(z)$ such that $f(z)=g(\upsilon(z))$ for two analytic functions $f$ and $g$, then we say $f$ is subordinate to $g$, written as $f \prec g$. Specifically, for the case when $g$ is univalent, we say $f \prec g$ if and only if $f(\mathbb{D}) \subset g(\mathbb
	{D})$ and $f(0)=g(0)$. In the recent past authors study certain type of inequalities in a complex plane (see~\cite{new3,rjm,new2}) and its applications. Moreover, many new subclasses of $\mathcal{S}^*$ were introduced and studied extensively (see~\cite{new,ron}). Motivated by these, we introduce another interesting subclass of $\mathcal{S}^*$:
	\begin{definition}
		For $f \in \mathcal{S}$, we define below a class $\mathcal{S}^*_{q}(\alpha)$ if $f$ satisfies:
		\begin{equation}\label{class}
			\RE\left(\dfrac{z f'(z)}{f(z)}\right) > \left|1+\dfrac{z f''(z)}{f'(z)} -\dfrac{z f'(z)}{f(z)}-\alpha\right| \quad (0 \leq \alpha <1).
		\end{equation}
	\end{definition}
	We observe that the class $\mathcal{S}_q^*(\alpha) \subseteq \mathcal{S}^*$. The identity function $f(z)=z$ satisfies the inequality \eqref{class} for all $0\leq \alpha <1$, hence $\mathcal{S}_q^*(\alpha) \neq \emptyset$. For $\alpha=0$, we set $\mathcal{S}_q^*(0)=:\mathcal{S}_q^*$. We also obtain $\mathcal{S}_q^*(\alpha)$ as a subclass of another fascinating class of starlike functions, which is introduced in the following section. The paper is structured as follows: In the next section, we show $f \in \mathcal{S}_q^*(\alpha)$ implies 
	$$ \dfrac{z f'(z)}{f(z)} \prec q_\alpha(z),$$ where $q_\alpha$ is given in~\eqref{subor}. This function $q_{\alpha}(z)$ which is brought to day light by studying the class~\eqref{class}, is first introduced by MacGregor~\cite{mac}. Using $q_{\alpha}(z)$, MacGregor~\cite{mac} also proved that convex functions of order $\alpha$ is starlike of order
	$$ \dfrac{2\alpha-1}{2-2^{2(1-\alpha)}}~(\alpha\neq 1/2) \text{ } \text{ and } \text{ } \dfrac{1}{\log 4}~(\alpha=1/2).$$ 
	Apart from this, he~\cite{mac} also proved $q_{\alpha}(z)$ to be univalent and
	$$ \min_{|z| <1} \RE q_{\alpha}(z) = q_\alpha(-1),$$ can refer~\cite{miler2} also. We study this function in detail and obtain certain sharp coefficient bounds for functions in $\mathcal{S}^*(q_\alpha)$.
	In the sequel we discuss many inclusion and radius results pertaining to the classes $\mathcal{S}^*_{q}(\alpha)$ and $\mathcal{S}^*(q_\alpha)$. 
	
	\section{Results related to $\mathcal{S}^*_q(\alpha)$}\label{sec2}
	We begin this section by showing the existence of some analytic function belonging to $ \mathcal{S}_q^*(\alpha)$ apart from $f(z)=z$.  
	\begin{example}
		Let $f_\gamma(z)=z+\gamma z^2$. If $|\gamma| < r_0$, where $r_0$ is the smallest such $r < 1$ satisfying the equation:
		\begin{align*}
			& \alpha^2 (2 r-1)^2 (r-1)^3 - 33 r^2 + 57 r^3 - 48 r^4 + 
			16 r^5 + 2 \alpha (r-1)^2 r (2 r-1) \\ & + 
			(2r-1)(r-1) (\alpha + r - 3 \alpha r + 2 \alpha r^2)(2-4r) + 
			9r =1\quad (0 \leq \alpha <1),
		\end{align*}
		then $f_\gamma \in \mathcal{S}^*_q(\alpha)$. Moreover, $r_0 \in(0,1/4]$.
	\end{example}
	\begin{proof}
		We know $f_\gamma \in \mathcal{S}$ whenever $|\gamma| \leq 1/2,$ thus $r_0 \in (0,1/2]$. A simple calculation yields
		\begin{equation*}
			\dfrac{z f_\gamma'(z)}{f_\gamma(z)} = \dfrac{1+ 2\gamma z}{1+\gamma z}\text{ and } 1+\dfrac{z f_\gamma''(z)}{f_\gamma'(z)}-\dfrac{z f_\gamma'(z)}{f_{\gamma}(z)}= \dfrac{\gamma z}{(1+\gamma z)(1+2 \gamma z)}.
		\end{equation*}
		Let $\gamma z= r e^{i \theta}$, $0 \leq  r \leq 1/2$ and $\theta \in [-\pi,\pi]$. Now, to prove the result, we show the following inequality holds:
		\begin{equation*}
			\RE\left(\dfrac{1+2 r e^{i \theta}}{1+r e^{i \theta}} \right) > \left|\dfrac{r e^{i \theta}}{(1+ r e^{i \theta})(1+2 r e^{i \theta})}-\alpha\right| \quad (0 \leq \alpha <1),
		\end{equation*}
		which is equivalent to show 
		\begin{align*}
			&- \sqrt{\dfrac{
					\alpha^2 + r^2 - 6 \alpha r^2 + 9 \alpha^2 r^2 + 4 \alpha^2 r^4 + 
					2 \alpha (-1 + 3 \alpha) r (1 + 2 r^2)x + 4 \alpha^2 r^2(2 x^2-1)}{
					1 + 9 r^2 + 4 r^4 + 6 (r + 2 r^3)x + 4 r^2(2 x^2-1)}} \\ & \quad + \dfrac{1 + 3 r x + 2 r^2}{1 + r^2 + 2 r x} =: h(\alpha,r,x)  > 0,
		\end{align*}
		where $x:= \cos \theta$. The function $h$ is increasing with respect to $x \in [-1,1]$, thus it suffices to show 
		\begin{align}\label{ex1}
			& \dfrac{2 r-1}{r-1}-\sqrt{\dfrac{\alpha^2 + r^2 - 6 \alpha r^2 + 13 \alpha^2 r^2 + 4 \alpha^2 r^4 -2 \alpha(-1 + 3 \alpha) r (1 + 2 r^2)}{1 + 13 r^2 + 4 r^4 - 6 (r + 2 r^3)}} \nonumber \\ \quad & =:h(\alpha,r)=h(\alpha,r,-1) > 0.
		\end{align}
		We observe that $h(\alpha,0) > 0$ and $h(\alpha,1/2) <0$. In fact, graphically we observe $h(\alpha,r) \leq 0$ for $r \geq 1/4$ for all $\alpha$. Since $h$ is a continuous function of $r$, thus there must exist $r_0 \in (0,1/4]$, which is a smallest positive root of $h(\alpha,r) =0$ by Intermediate value property. Consequently \eqref{ex1} holds for $ 0 < r < r_0$. Therefore, we conclude that $f_\gamma \in \mathcal{S}^*_q(\alpha)$ whenever $|\gamma| < r_0.$
	\end{proof}
	\begin{theorem}\label{firstthm}
		Let $f \in \mathcal{S}_q^*(\alpha)$ for $0\leq \alpha<1$. Then
		\begin{equation}\label{subor}
			\dfrac{z f'(z)}{f(z)} \prec q_\alpha(z):= \begin{cases}		\dfrac{(1-2 \alpha)z}{(1-z)(1-(1-z)^{1-2 \alpha})}, & \alpha \neq 1/2  \\ &  \\ \dfrac{-z}{(1-z)\log(1-z)}, & \alpha = 1/2. 
			\end{cases}
		\end{equation}
	\end{theorem}
	\begin{proof}
		Let us define \begin{equation*}
			p(z) = \dfrac{z f'(z)}{f(z)},
		\end{equation*}
		then equation \eqref{class} reduces to
		\begin{align*}
			\RE p(z) > & \left|\dfrac{z p'(z)}{p(z)}-\alpha\right|\\ \geq & \RE\left(\alpha-\dfrac{z p'(z)}{p(z)}\right),
		\end{align*}
		which yields
		\begin{equation*}
			\RE\left(p(z)+\dfrac{z p'(z)}{p(z)}\right) > \alpha.
		\end{equation*}
		Using the subordination concept, we write the above inequality as follows:
		\begin{equation}\label{subord}
			p(z)+\dfrac{z p'(z)}{p(z)} \prec \dfrac{1+(1-2 \alpha)z}{1-z} \quad (0 \leq \alpha<1).
		\end{equation}
		From~\cite[Theorem 3.3 d, p.~109]{miler2}, we have:
		\begin{equation}\label{dom}
			q(z) + \dfrac{z q'(z)}{q(z)} = \dfrac{1+(1-2 \alpha)z}{1-z},
		\end{equation}
		where $q$ is the best dominant of the subordination \eqref{subord}. On solving equation \eqref{dom}, we obtain
		\begin{equation*}
			q(z) =  \left(\dfrac{z}{(1-z)^{2(1-\alpha)}}\right)\left(\int_{0}^{z}(1-t)^{-2(1-\alpha)} dt\right)^{-1} \quad (\alpha \neq 1/2)
		\end{equation*}
		and for $\alpha=1/2$
		\begin{equation*}
			q(z)= \left(\dfrac{z}{1-z}\right)\left(\int_{0}^{z}\dfrac{dt}{1-t}\right)^{-1}.
		\end{equation*}
		Hence the proof follows.
	\end{proof}
	We consider a function $f \in \mathcal{A}$, then 
	\begin{equation} \label{z/f}
		\dfrac{z}{f'(z)}=:z+\sum_{n=2}^{\infty}c_n z^n\text{ and } \dfrac{z}{f(z)}=:1+\sum_{n=1}^{\infty}b_nz^n,
	\end{equation}where $b_n \in \mathbb{R}$ and $c_n \geq 0$.
	\begin{theorem}
		Let $f \in \mathcal{S}$ of the form~\eqref{z/f} belongs to $ \mathcal{S}^*_{q}(\alpha)$. Then 
		\begin{equation*}
			\sum_{n=2}^{\infty}(n+\alpha-2) c_n < (1-\alpha).
	\end{equation*}\end{theorem}
	\begin{proof}
		A simple computation yields 
		\begin{equation*}
			z\left(\dfrac{z}{f(z)}\right)'=\dfrac{z}{f(z)}-\left(\dfrac{z}{f(z)}\right)^2f'(z)\text{ and } z\left(\dfrac{z}{f'(z)}\right)'=\dfrac{z}{f'(z)}-\left(\dfrac{z}{f'(z)}\right)^2f''(z).
		\end{equation*}
		Also, $f \in \mathcal{S}^*_{q}(\alpha)$ of the form~\eqref{z/f} gives 
		\begin{equation*}
			\left|1+\dfrac{z f''(z)}{f'(z)} -\dfrac{z f'(z)}{f(z)}-\alpha\right| < \RE\left(\dfrac{z f'(z)}{f(z)}\right)
		\end{equation*}
		if and only if
		\begin{equation*}
			\left|1-\dfrac{z\left(\dfrac{z}{f'(z)}\right)'}{\dfrac{z}{f'(z)}}+\dfrac{z \left(\dfrac{z}{f(z)}\right)'}{\dfrac{z}{f(z)}} -\alpha\right| < \RE \dfrac{\dfrac{z}{f(z)}-z\left(\dfrac{z}{f(z)}\right)'}{\dfrac{z}{f(z)}}\end{equation*}
		if and only if
		\begin{equation*}
			\left|1-\dfrac{1+\sum_{n=2}^{\infty}nc_n z^{n-1}}{1+\sum_{n=2}^{\infty}c_n z^{n-1}}+\dfrac{\sum_{n=1}^{\infty}nb_n z^n}{1+\sum_{n=1}^{\infty}b_n z^n}-\alpha\right| < \RE\left(1-\dfrac{\sum_{n=1}^{\infty}nb_n z^n}{1+\sum_{n=1}^{\infty}b_n z^n}\right).
		\end{equation*}
		Now if $z \in \mathbb{D}$ is real and tends to $1^-$ through reals, then from the above inequality we deduce
		\begin{equation*}
			\dfrac{1+\sum_{n=2}^{\infty}n c_n}{1
				+\sum_{n=2}^{\infty}c_n} < 2-\alpha.
		\end{equation*}
		Therefore the results follows now.
	\end{proof}
	With the advent of Theorem~\ref{firstthm}, we define another interesting subclass of $\mathcal{S}^*$ as follows:
	\begin{definition}\label{sq}
		Let $\mathcal{S}^*(q_\alpha)$ denote the class of analytic functions $f \in \mathcal{S}$, satisfying
		\begin{equation}\label{qa}
			\dfrac{z f'(z)}{f(z)} \prec q_\alpha(z) \quad (z \in \mathbb{D},~0\leq \alpha<1).\end{equation}
	\end{definition}
	We infer from Theorem~\ref{firstthm}, $\mathcal{S}_q^*(\alpha) \subset \mathcal{S}^*(q_\alpha)$, so is non-empty. This class generalizes subclass of $\mathcal{S}^*$, such as for $\alpha=0$, it reduces to $\mathcal{S}^*(1/2)$.
	
	\section{About $\mathcal{S}^*(q_\alpha)$}\label{sec3}
	A function $f \in \mathcal{S}^*(q_\alpha)$ if and only if there exists an analytic function $s$, $s(z) \prec q_{\alpha}(z)$ such that
	\begin{equation}\label{q}
		f(z) = z \exp\int_{0}^{z} \dfrac{s(t)-1}{t} dt.
	\end{equation}
	Specifically, for $s(z) = q_\alpha(z)$, the structural formula~\eqref{q} yields
	\begin{align}\label{falpha}
		f_\alpha(z)  = \begin{cases} \dfrac{1-(1-z)^{2\alpha-1}}{2\alpha-1}, & \alpha \neq 1/2; \\ & \\ -\log(1-z), & \alpha=1/2.\end{cases}
	\end{align}
	Taylor series of $f_\alpha(z)$ $(0 \leq \alpha <1)$ is given as follows \begin{align*}f_\alpha(z) &= z+ (1-\alpha) z^2 + (3 - 5 \alpha + 2 \alpha^2) \dfrac{z^3}{3} + (6 - 13 \alpha + 9 \alpha^2 - 2 \alpha^3) \dfrac{z^4}{6}+\cdots  \\ &=  z+ \sum_{n=2}^{\infty}\left(\dfrac{\prod_{j=2}^{n}(j-2\alpha)}{n!}z^n\right) \quad (0 \leq \alpha <1),\end{align*} plays an extremal function for many cases in $\mathcal{S}^*(q_\alpha)$. Interestingly, this function $f_\alpha(z)$ is also the extremal function for the class $\mathcal{K}(\alpha)$ see~\cite{duren,miler2}. 
	\begin{figure}[ht]
		\begin{minipage}{0.45\textwidth}
			\includegraphics[width=0.65\linewidth, height=4.5cm]{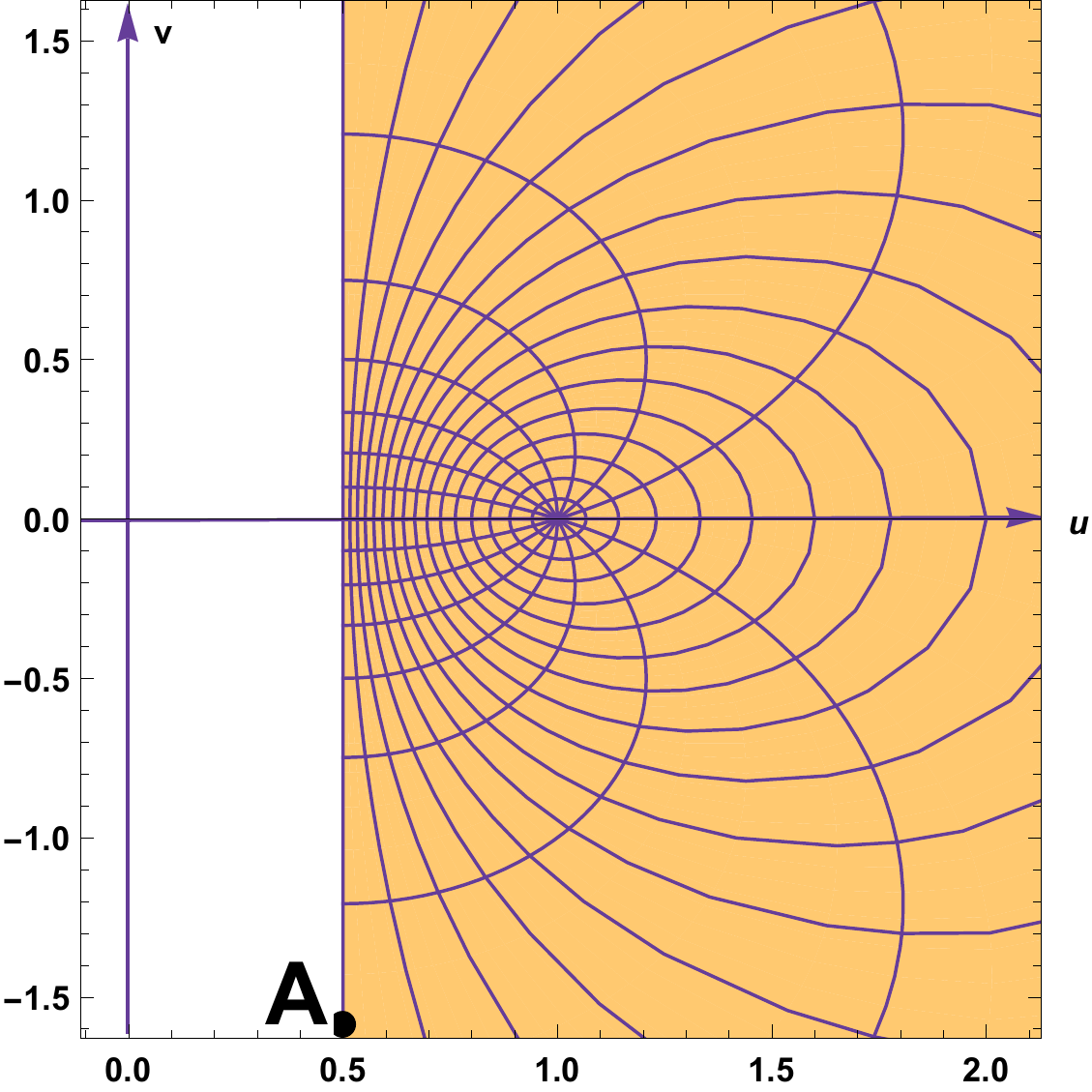} \\ $\alpha=0$,~$\textbf{A}=0.5$\end{minipage}
		\begin{minipage}{0.45\textwidth}
			\includegraphics[width=0.65\linewidth, height=4.5cm]{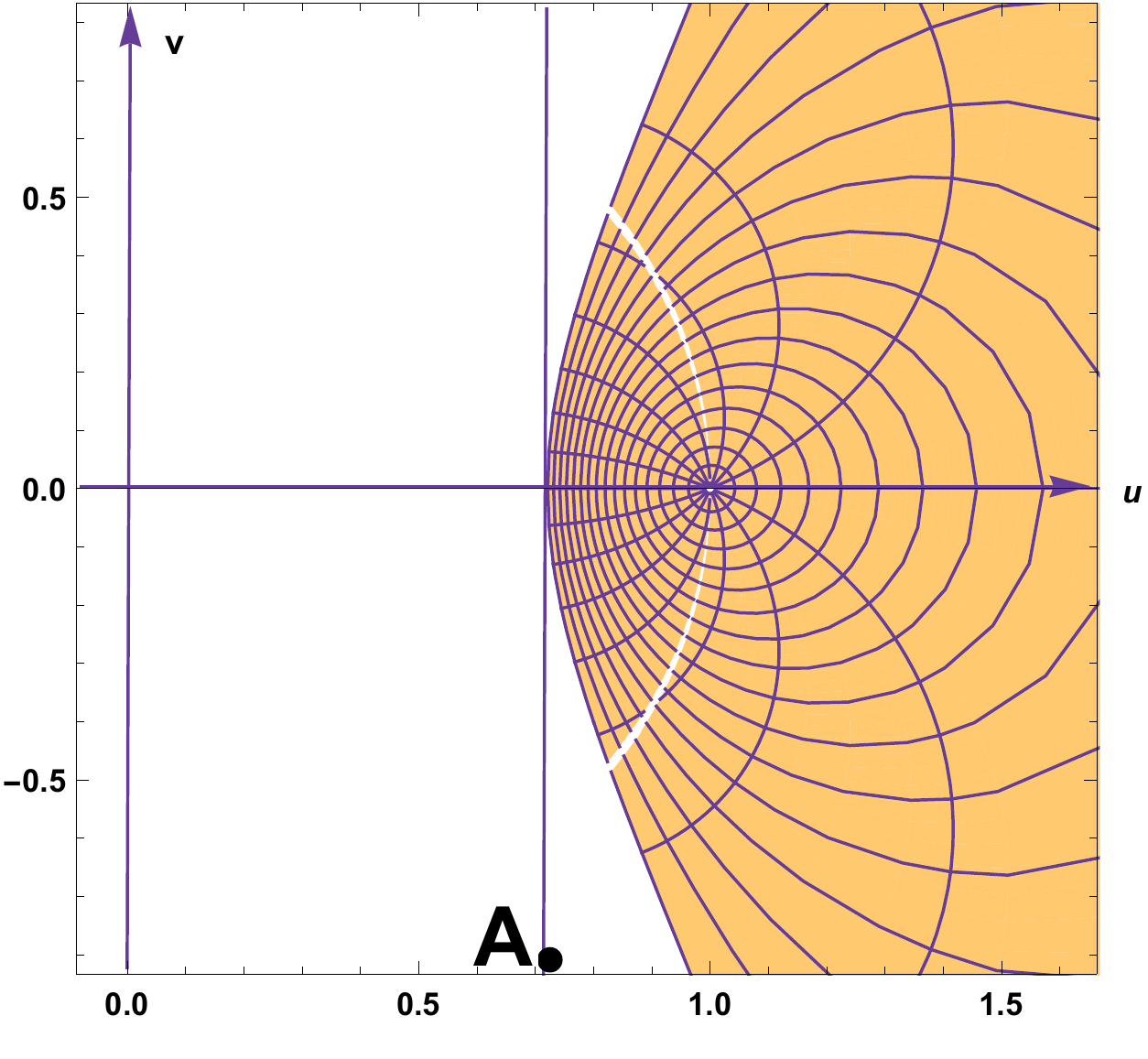}\\ $\alpha=1/2$,~$\textbf{A}= 0.721348$\end{minipage}
	\end{figure}
	\begin{figure}[ht]
		\begin{minipage}{0.45\textwidth}
			\includegraphics[width=0.65\linewidth, height=4.5cm]{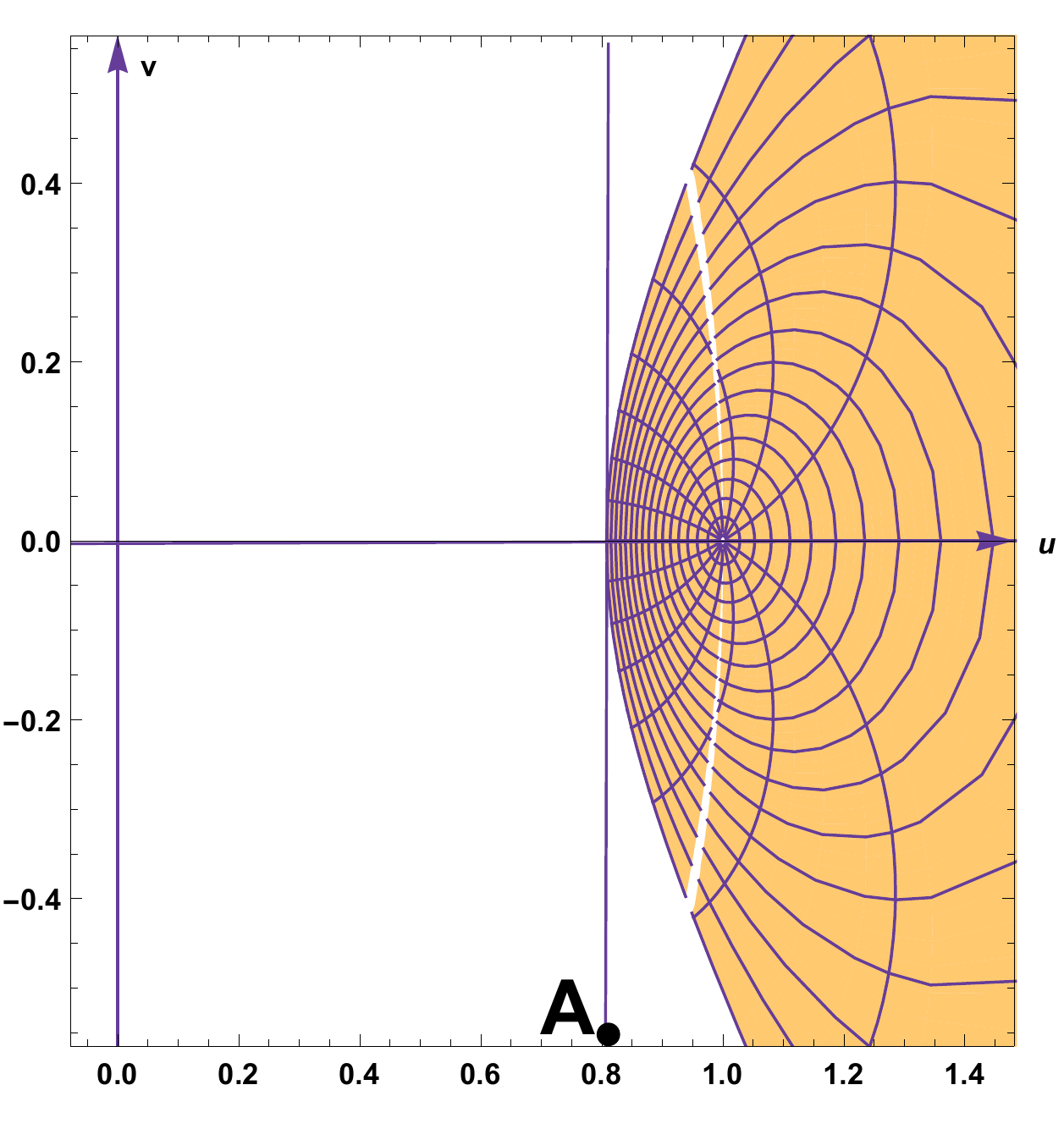} \\$\alpha=2/3,$~$\textbf{A}=0.807887$
		\end{minipage}
		\begin{minipage}{0.45\textwidth}
			\includegraphics[width=0.65\linewidth, height=4.5cm]{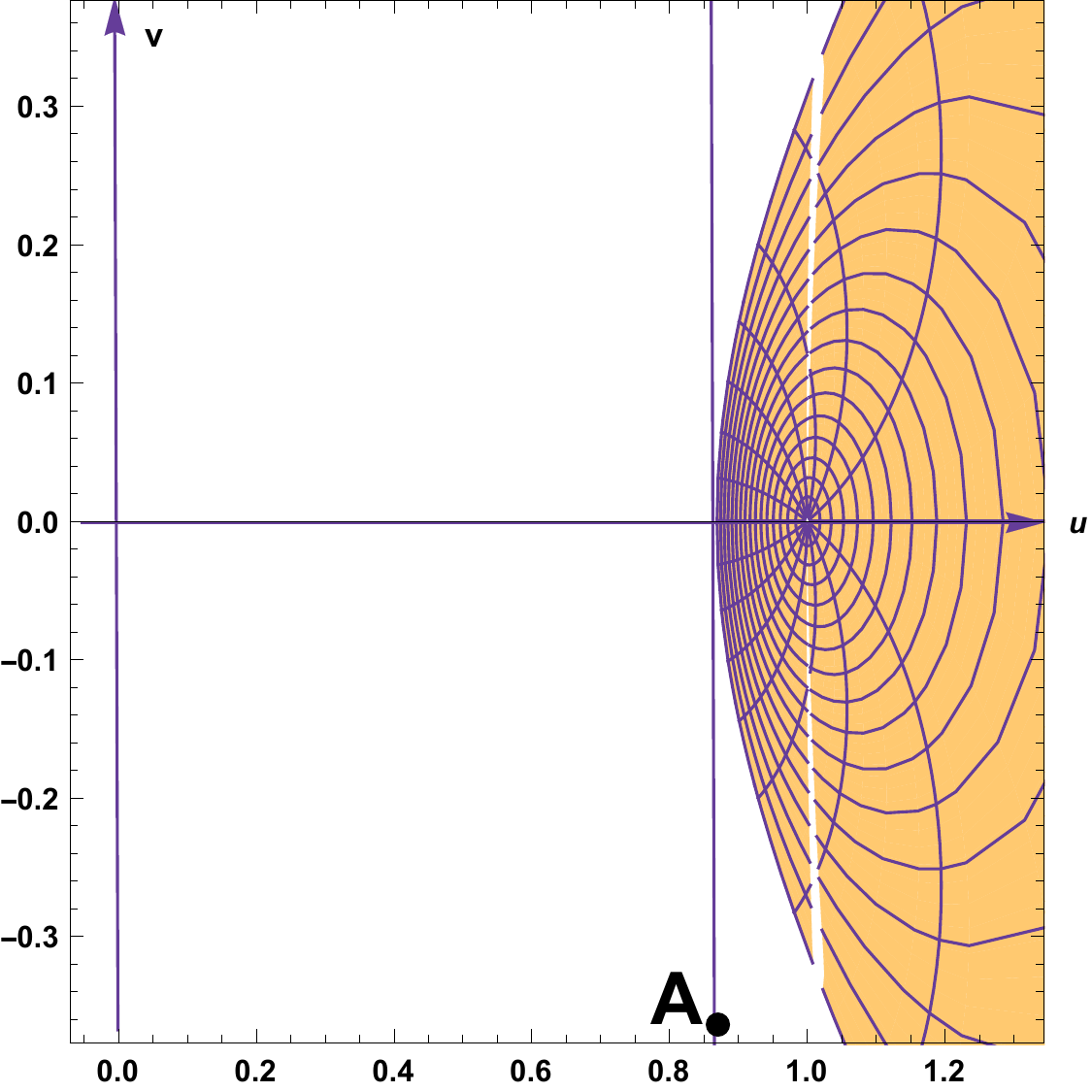}\\ $\alpha=7/9$,~$\textbf{A}=0.869128$
		\end{minipage}
		\caption{Image of unit disk $\mathbb{D}$ under the function $q_\alpha(z)$ for various $\alpha$.}
		\label{fig:incl_rel}
	\end{figure}
	\begin{remark}\label{cor1}
		Using the results in the proof of Theorem~\ref{firstthm}, we deduce 
		$$ p(z)+\dfrac{z p'(z)}{p(z)} \prec \dfrac{1+(1-2 \alpha)z}{1-z} \Rightarrow p(z) \prec q_\alpha(z),$$ which further yields 
		\begin{equation}\label{convexsub} 1+\dfrac{z f''(z)}{f'(z)} \prec \dfrac{1+(1-2\alpha)z}{1-z} \Rightarrow \dfrac{z f'(z)}{f(z)} \prec q_\alpha(z).\end{equation}
		The function $f_\alpha(z)$, given by~\eqref{falpha} is an extremal function for the above differential subordination implication. This result~\eqref{convexsub} is initially proved by MacGregor~\cite{mac} and later in~\cite[p.~113--115]{miler2}, using the following :
		\begin{equation}\label{minq} \min_{|z|<1}\left(\RE q_\alpha(z)\right) = q_{\alpha}(-1) = \begin{cases}
				\dfrac{2 \alpha-1}{2-2^{2(1-\alpha)}}, & \alpha \neq 1/2 \\ & \\ \dfrac{1}{\log 4}, & \alpha=1/2,\end{cases}\end{equation} and 
		$\min_{|z|=r} \RE q_\alpha(z) = q_\alpha(-r)$, $\max_{|z|=r} \RE q_{\alpha}(z) = q_\alpha(r).$  
	\end{remark}  
	We first show that an analytic univalent function $q_\alpha(z)$ is a Ma-Minda function. We have $q_{\alpha}(0)=1$ and $\RE q_{\alpha}(z) >0$ $(0 \leq \alpha <1)$. A further calculation reveals $q_{\alpha}'(0) >0$. Now the Taylor series of $q_{\alpha}(z)$ is given as follows:
	$$ q_\alpha(z) = 1 + (1 - \alpha) z +  (3 - 4\alpha + \alpha^2) \dfrac{z^2}{3} +  (2 - 3\alpha + \alpha^2) \dfrac{z^3}{2} + 
	(45 - 72\alpha + 26 \alpha^2 + 2 \alpha^3 - \alpha^4) \dfrac{z^4}{45}+\cdots,$$
	shows the function is symmetric with respect to the real axis as it has real coefficients. For detail analysis for the geometry of functions defined on $\mathbb{D}$, see~\cite{banga}. 
	The function $q_\alpha(z)$ is starlike with respect to $q_\alpha(0)=1$ as we get
	$$\RE\left(\dfrac{e^{i \theta} q_{\alpha}'(e^{i \theta})}{q_{\alpha}(e^{i \theta})-1}\right) >0,\qquad(-\pi \leq  \theta \leq \pi) \quad (0 \leq \alpha <1),$$ by performing a highly complex computation in Mathematica~11.0, which otherwise is not possible manually. Therefore $q_\alpha(z)$ is a Ma-Minda function as it satisfies all the conditions to be one. Figure~1 depicts $q_{\alpha}(\mathbb{D})$ for various $\alpha \in [0,1)$.\\ \ \\
	\textbf{Observation:} It is clear that $q_{0}(z)$ is a convex function but it is not the case for every $q_{\alpha}$. In fact, using Mathematica~11.0, graph of $\RE \left(1+\tfrac{z q_{\alpha}''(z)}{q_{\alpha}'(z)}\right)\bigg|_{z=e^{i \theta}}$ $(-\pi \leq \theta \leq \pi)$, is not positive for some $\alpha \in (0,1)$. Thus we coin below a problem which is open at present.\\ \ \\
	\textbf{Open Problem:} Find the range of $\alpha$ in $(0,1)$ for which $q_\alpha(\mathbb{D})$ is convex.\\ \ \\  
	Based on certain subordination results proved in~\cite{ma}, we have $f(z)/z \prec f_\alpha(z)/z$ and the following:
	\begin{theorem}
		Let $f \in \mathcal{S}^*(q_\alpha)$ and $|z_0|=r<1.$ Then
		\begin{enumerate}
			\item[\emph{(i)}] (Growth Theorem): $-f_\alpha(-r) \leq |f(z_0)| \leq f_\alpha(r)$
			$$ \dfrac{(1+r)^{2\alpha-1}-1}{2\alpha-1} \leq |f(z_0)| \leq \dfrac{1-(1-r)^{2\alpha-1}}{2\alpha-1}~(\alpha\neq 1/2)\text{; } \log(1+r) \leq |f(z_0)| \leq -\log(1-r)~(\alpha=1/2).$$
			\item[\emph{(ii)}] (Distortion Theorem): $f_\alpha'(-r) \leq |f'(z_0)| \leq f_\alpha'(r)$
			$$ (1+r)^{2(\alpha-1)} \leq |f(z_0)| \leq  (1-r)^{2(\alpha-1)}.$$
			\item[\emph{(iii)}] (Rotation Theorem): $|\arg(f(z_0)/z_0)| \leq \max_{|z|=r} \arg(f_\alpha(z)/z)$.
		\end{enumerate}	
		Equality holds at some $z_0 \neq 0$ if and only if $f$ is a rotation of $f_\alpha.$
	\end{theorem} 
	\begin{remark}
		Interestingly, the above Growth, Distortion and Rotation results hold in the case of $\mathcal{K}(\alpha)$~\cite[Theorem~1, p.~139]{bieberbach1983} as well. It is due to the fact that the classes $\mathcal{S}^*(q_{\alpha})$ and $\mathcal{K}(\alpha)$ have the same extremal function $f_{\alpha}(z)$.  
	\end{remark}
	\subsection{Coefficient Estimates}
	
	From~\cite{banga}, we get $|a_2| \leq 1-\alpha$ for functions in $\mathcal{S}^*(q_\alpha)$ and the extremal function is $f_{\alpha}(z)$, defined in~\eqref{falpha}. Using Fekete-Szeg\"{o} bounds for Ma-Minda class~\cite{murg} and \cite[Remark 4.1, p.~12]{banga}, we obtain the following:
	\begin{theorem}\label{fekcor}
		Let $f \in \mathcal{S}^*(q_\alpha)$. Then we have
		\begin{equation*}|a_3-t a_2^2| \leq \begin{cases}
				\dfrac{(3-2\alpha)(1-\alpha)}{3}-t(1-\alpha)^2, & t \leq \dfrac{3-4\alpha}{6(1-\alpha)} \\ \dfrac{(1-\alpha)}{2}, & \dfrac{3-4\alpha}{6(1-\alpha)} \leq t \leq \dfrac{9-4\alpha}{6(1-\alpha)} \\ t(1-\alpha)^2-\dfrac{(3-2\alpha)(1-\alpha)}{3}, & t \geq \dfrac{9-4\alpha}{6(1-\alpha)}.
		\end{cases}\end{equation*} The result is sharp.
	\end{theorem}
	Taking $t=0$ and $1$, respectively in the above result, we get the following 
	\begin{corollary}
		Let $f \in \mathcal{S}^*(q_\alpha)$. Then we have\begin{itemize}
			\item[(i)] $|a_3| \leq \begin{cases}
				\dfrac{(2\alpha-3)(\alpha-1)}{3},  & \alpha \leq 3/4 \\ 
				\dfrac{1-\alpha}{2}, & \alpha \geq 3/4.
			\end{cases}$\\ The result is sharp and the extremal function is given by $f_\alpha(z)$, given in~\eqref{falpha} for $\alpha \leq 3/4$ and  $\tilde{f}_{\alpha}(z)= \sqrt{\tfrac{(1-z^2)^{2a-1}-1}{2a-1}}$ for $\alpha \geq 3/4$.\\
			\item[(ii)] $|a_3-a_2^2| \leq \tfrac{1-\alpha}{2}$. The result is sharp for $\tilde{f}_{\alpha}(z)$, defined in the above part \emph{(i)} for  $\alpha \neq 1/2$ and $f_{1/2}(z):=\sqrt{\log(1-z^2)}$ for $\alpha=1/2$.\end{itemize}
	\end{corollary}
	Recall that inverse of a function $f(z)=z+\sum_{n=2}^{\infty}a_n z^n \in \mathcal{S}$ is given by $f^{-1}(f(z))=z$, $z \in \mathbb{D}$ and $f(f^{-1}(w))=w$ $(|w|< r_0$, $r_0 >1/4)$, for which 
	\begin{equation}\label{finv} f^{-1}(w)= w - a_2 w^2 - (2a_2^2
		- a_3)w^3 - (5a_3^2
		- 5a_2a_3 + a_4)w^4 + \cdots.\end{equation}
	\begin{corollary}
		Let $f \in \mathcal{S}^*(q_\alpha)$ and the inverse be given by $f^{-1}(w)= w+\sum_{n=2}^{\infty} b_n w^n$. Then \begin{itemize}
			\item[(i)] $|b_2| \leq 1-\alpha,$\\
			\item[(ii)]	$|b_3| \leq \begin{cases}
				2(1-\alpha)^2-\dfrac{(3-2\alpha)(1-\alpha)}{3}, & 0 \leq \alpha \leq 3/8 \\
				\dfrac{1-\alpha}{2}, & 3/8 \leq \alpha \leq 1.
			\end{cases}$\end{itemize} The inequalities are sharp. \end{corollary}
	\begin{proof}
		Now, comparing the Taylor series of $f^{-1}$ given in~\eqref{finv} and in the hypothesis, we get $|b_2| = |a_2|\leq (1-\alpha)$ and $f_{\alpha}(z)$ acts as an extremal for this bound. We have
		$|b_3|= |2 a_2^2-a_3|.$ Now the result follows from Theorem~\ref{fekcor} when $t=2$. 
	\end{proof}
	The logarithmic coefficients $\beta_n$ for functions $f \in \mathcal{A}$ is given by:
	\begin{equation*}
		\log \dfrac{f(z)}{z} = \sum_{n=1}^{\infty} 2 \beta_n z^n \quad (z \in \mathbb{D}).
	\end{equation*}
	In the literature, the sharp bound $|\beta_n| \leq 1/n$ $(n \geq 1)$ for functions in $\mathcal{S}^*$ is already proved and equality holds for the Koebe function. Here we derive the bounds of $|\beta_n|$ for functions in the class $\mathcal{S}^*(q_\alpha)$. 
	For this, let us define a set
	$$E(\alpha)=\{\alpha \in [0,1): q_\alpha(\mathbb{D})\text{ is } \text{ convex }\}.$$
	\begin{theorem}
		Let $f \in \mathcal{S}^*(q_\alpha)$ $(\alpha \in E(\alpha))$. Then, we have
		$$ |\beta_n| \leq \dfrac{1-\alpha}{2n}.$$ The inequality is sharp.
	\end{theorem}
	\begin{proof}
		For $f \in \mathcal{S}^*(q_\alpha)$, we have, 
		\begin{equation*}
			\dfrac{zf'(z)}{f(z)}\prec q_\alpha(z),\end{equation*}
		equivalent to
		$$ z\left(\log\dfrac{f(z)}{z}\right)'\prec q_{\alpha}(z)-1,$$
		which further yields
		$$\sum_{n=1}^{\infty} 2 n\beta_n z^n  \prec \sum_{n=1}^{\infty} B_n z^n := q_\alpha(z)-1.$$ Applying Rogosinski's result~\cite{rog}, we obtain
		$$ 2 n |\beta_n| \leq |B_1|=: 1-\alpha.$$Hence the result.
	\end{proof}
	
	Let $g(z)=z+\sum_{n=2}^{\infty} g_n z^n$ and $h(z)=z+\sum_{n=2}^{\infty}h_n z^n$ be two analytic functions in $\mathbb{D}$, then convolution of $g$ and $h$ is defined as follows $$(g*h)(z)= z+\sum_{n=2}^{\infty} g_n h_n z^n.$$ 
	The following result gives the necessary and sufficient condition in terms of convolution for the functions in $\mathcal{S}^*(q_\alpha)$.
	\begin{theorem}
		A function $f \in \mathcal{S}$ belongs to $\mathcal{S}^*(q_\alpha)$ if and only if it satisfies
		\begin{equation}\label{convres}
			\dfrac{1}{z}\left(f(z) * \dfrac{z-\lambda z^2}{(1-z)^2}\right) \neq 0 \quad (z \in \mathbb{D}),
		\end{equation}
		where $$\lambda:= \lambda(\theta)=\begin{cases} \dfrac{1-2 \alpha}{(1-2\alpha)-(e^{-i\theta}-1)(1-(1-e^{i \theta})^{1-2\alpha})}, & \alpha\neq 1/2\\ & \\ \left(1-(1-e^{- i\theta})\log(1-e^{i \theta})\right)^{-1}, & \alpha=1/2,\end{cases}~~~~~~(\theta \in [-\pi,\pi]).$$
	\end{theorem}
	\begin{proof}
		Let $\alpha \neq 1/2$. Then from Definition~\ref{sq}, we have $f \in \mathcal{S}^*(q_\alpha)$ if and only if it satisfies the subordination~\eqref{qa}, 
		equivalently
		$$	\dfrac{z f'(z)}{f(z)} \neq  \dfrac{(1-2 \alpha)e^{i \theta}}{(1-e^{i \theta})(1-(1-e^{i \theta})^{1-2 \alpha})} \quad (z \in \mathbb{D}, \theta \in [-\pi,\pi]),$$
		further which yields
		$$f'(z) \neq  \left(\dfrac{(1-2 \alpha)e^{i \theta}}{(1-e^{i \theta})(1-(1-e^{i \theta})^{1-2 \alpha})}\right) \dfrac{f(z)}{z}.$$ Also, it can be expressed as
		\begin{equation}\label{conv}\dfrac{1}{z}\left(z f'(z)- \dfrac{(1-2 \alpha)e^{i \theta}}{(1-e^{i \theta})(1-(1-e^{i \theta})^{1-2 \alpha})} f(z)\right) \neq 0.\end{equation}
		As we recall 
		$$ zf'(z) = f(z) * \dfrac{z}{(1-z)^2}\text{ and } f(z) = f(z) * \dfrac{z}{1-z},$$
		\eqref{conv} becomes
		$$ \dfrac{1}{z} \left( f(z) * \dfrac{z-\tfrac{(1-2 \alpha) e^{i \theta}}{(1-e^{i \theta})(1-(1-e^{i \theta})^{1-2\alpha})}(z-z^2)}{(1-z)^2}\right) \neq 0,$$
		which upon simplification reduces to~\eqref{convres}. Similarly the result follows for $\alpha=1/2$.  
	\end{proof}
	Now, using the convolution properties, we have
	$$ \dfrac{1}{z}\left(f(z) * \dfrac{z-\lambda z^2}{(1-z)^2}\right)  = \dfrac{1}{z}\left((1-\lambda)z f'(z)+\lambda f(z)\right).$$ Substituting the above equation in~\eqref{convres} along with the Taylor series expansion of $f \in \mathcal{A}$, we deduce the following result: 
	\begin{corollary}
		A function $f \in \mathcal{S}$ belongs to $\mathcal{S}^*(q_\alpha)$ if and only if it satisfies\begin{equation*} 1 \neq \sum_{n=2}^{\infty} \rho(\theta)a_n z^{n-1}:=
			\begin{cases}	\sum_{n=2}^{\infty} \dfrac{(e^{-i\theta}-1)(1-(1-e^{i \theta})^{1-2\alpha})n-(1-2\alpha)}{(1-2\alpha)-(e^{-i\theta}-1)(1-(1-e^{i \theta})^{1-2\alpha})}a_n z^{n-1}, &\alpha\neq 1/2 \\ & \\ \sum_{n=2}^{\infty} \dfrac{(1-e^{-i\theta})\log(1-e^{i \theta})n-1}{1-(1-e^{-i\theta})\log(1-e^{i \theta})} a_n z^{n-1}, &\alpha=1/2.\end{cases} 
		\end{equation*}
	\end{corollary}
	Consequently, the next result follows. 
	\begin{corollary}\label{concor}
		Let $f \in \mathcal{S}$ satisfies
		\begin{equation*}
			\sum_{n=2}^{\infty} |\rho(\theta)||a_n| < 1,\end{equation*} then $f \in \mathcal{S}^*(q_\alpha)$.
	\end{corollary}
	\section{Inclusion Relation}\label{sec4}
	We prove below some inclusion relations of $\mathcal{S}_q^*(\alpha)$.
	\begin{theorem}\label{convex}
		Let $0 \leq \alpha <1$, then $\mathcal{S}_q^*(\alpha) \subset \mathcal{K}(\alpha)$.
	\end{theorem}
	\begin{proof}
		Let $f \in  \mathcal{S}_q^*(\alpha)$, then inequality~\eqref{class} yields
		\begin{equation*}
			\RE\left(\dfrac{z f'(z)}{f(z)} \right) > - \RE\left(1+\dfrac{z f''(z)}{f'(z)}-\dfrac{z f'(z)}{f(z)}-\alpha\right),
		\end{equation*}  
		which upon a straightforward calculation yields
		\begin{equation*}
			\RE\left(1+\dfrac{z f''(z)}{f'(z)}\right) > \alpha.
		\end{equation*}
		This completes the proof.
	\end{proof}
	Silverman~\cite{silver} introduced the class $$\mathcal{G}_b:=\left\{f \in \mathcal{A}: \left|\dfrac{1+z f''(z)/f'(z)}{z f'(z)/f(z)}-1\right| <b,~0<b \leq 1\right\}.$$  In the following result we show the relation between the classes $\mathcal{G}_b$ and $\mathcal{S}_q^*$.
	\begin{corollary}
		$\mathcal{S}^*_q \subset \mathcal{G}_1$. 
	\end{corollary}
	\begin{proof}
		Since $\mathcal{S}^*_q:=\mathcal{S}_q^*(0)$, we have $\alpha=0$. Let $f \in \mathcal{S}^*_q$, then inequality~\eqref{class} becomes 
		\begin{equation*}
			\RE\left(\dfrac{z f'(z)}{f(z)}\right) > \left|1+\dfrac{z f''(z)}{f'(z)} -\dfrac{z f'(z)}{f(z)}\right|,
		\end{equation*}
		which implies
		\begin{align*}
			\left|\dfrac{z f'(z)}{f(z)}\right| > \left|1+\dfrac{z f''(z)}{f'(z)} -\dfrac{z f'(z)}{f(z)}\right|,\end{align*} therefore, we have \begin{align*} \left|\left(\dfrac{1+z f''(z)/f'(z)}{z f'(z)/f(z)}\right)-1\right| < 1.
		\end{align*}
		Hence the result.
	\end{proof}
	Mocanu~\cite{moc} studied the class comprising of the linear combination of analytic representations of convex and starlike functions and defined it as $\rho-$ convex class as follows:
	\begin{equation*}
		\RE(\rho(1+z f''(z)/f'(z))+(1-\rho)(z f'(z)/f(z))) \geq 0 \quad (\rho \in \mathbb{R}).
	\end{equation*}
	\begin{corollary}
		Let $f \in \mathcal{S}^*_q$, then $f$ is $-1$-convex function.
	\end{corollary}
	\begin{proof}
		Let $f \in \mathcal{S}^*_q$, then we have
		\begin{align*}
			&\RE\left(\dfrac{z f'(z)}{f(z)}\right) > \RE\left(1+\dfrac{z f''(z)}{f'(z)} -\dfrac{z f'(z)}{f(z)}\right), \\  &~ \RE\left(2\dfrac{z f'(z)}{f(z)}- \left(1+\dfrac{z f''(z)}{f'(z)}\right)\right) > 0.
		\end{align*}
		This completes the proof.
	\end{proof}
	The known inclusion $\mathcal{S}_q^*(\alpha) \subset \mathcal{S}^*(q_\alpha)$ and~\eqref{minq} yield the following result.
	\begin{theorem}\label{inc}
		We have $\mathcal{S}_q^*(\alpha) \subset \mathcal{S}^*(q_\alpha) \subset \mathcal{S}^*(\gamma)$ whenever $0\leq \gamma\leq \tfrac{2\alpha-1}{2(1-2^{1-2\alpha})}$ for $\alpha \neq 1/2$, or whenever $0 \leq \gamma \leq 1/\log 4$ for $\alpha=1/2$. 
	\end{theorem}
	\begin{theorem}\label{gammathm}
		Let $f \in \mathcal{S}^*(q_\alpha)$. Then we have
		\begin{equation}\label{gamma}
			\RE \dfrac{f(z)}{z} > \gamma(\alpha):=\begin{cases}\dfrac{3-2\alpha-2^{2(1-\alpha)}}{4-2\alpha-3.2^{1-2\alpha}}, & \alpha \neq 1/2 \\ & \\  \dfrac{2(1-\log 4)}{2-3 \log 4}, & \alpha = 1/2,\end{cases}
		\end{equation}
		which also implies $\mathcal{S}_q^*(\alpha) \subset \mathcal{S}^*(q_\alpha) \subset \mathcal{Q}(\gamma(\alpha))$.
	\end{theorem}
	\begin{proof}
		For brevity, let us denote $\gamma(\alpha)=:\gamma$. We observe that $0<\gamma \leq 1/2$ for $0\leq \alpha<1$. Let us define an analytic function $p$ with $p(0)=1$ as
		\begin{equation*}
			p(z)=\dfrac{1}{1-\gamma}\left(\dfrac{f(z)}{z}-\gamma\right),	
		\end{equation*}
		which upon simplification gives
		\begin{equation*}
			\dfrac{z f'(z)}{f(z)} =1+\dfrac{(1-\gamma)zp'(z)}{(1-\gamma)p(z)+\gamma}=:\xi(p(z),zp'(z)),
		\end{equation*}
		where $$\xi(r,s):= 1+\dfrac{(1-\gamma)s}{(1-\gamma)r+\gamma}.$$ Let $\alpha':= \tfrac{2\alpha-1}{2(1-2^{1-2\alpha})}$ $(\alpha
		\neq 1/2)$ and otherwise $\alpha':=\tfrac{1}{\log 4}$. As $f \in \mathcal{S}^*(q_\alpha)$, then Theorem~\ref{inc} yields  $$\xi(p(z),zp'(z))\subset \{w \in \mathbb{C}: \RE w> \alpha'\} =:\Omega_{\alpha'}.$$ Let $\lambda, \tau \in \mathbb{R}$ be such that $\lambda \leq \tfrac{-1}{2}(1+\tau^2)$. Then we have
		\begin{align}\label{2}
			\RE\{\xi(i \tau,\lambda)\}&= \RE\left(1+\dfrac{\lambda(1-\gamma)}{(1-\gamma) i\tau +\gamma}\right)\nonumber \\&=1+\dfrac{\lambda \gamma (1-\gamma)}{(1-\gamma)^2\tau^2+\gamma^2} \nonumber\\&\leq 1-\left(\dfrac{\gamma(1-\gamma)}{2}\right)\left(\dfrac{(1+\tau^2)}{(1-\gamma)^2\tau^2+\gamma^2}\right)=1-\dfrac{\gamma(1-\gamma)}{2} h(\tau),
		\end{align} 
		where $h(\tau):=\dfrac{(1+\tau^2)}{(1-\gamma)^2\tau^2+\gamma^2}$, is a decreasing function of $\tau$ for all $\gamma$, we deduce the following
		\begin{equation*}
			\dfrac{1}{(1-\gamma)^2}\leq h(\tau^2) \leq \dfrac{1}{\gamma^2}. 
		\end{equation*} 
		Using the above inequality in~\eqref{2}, we obtain
		\begin{equation*}
			\RE\{\xi(i \tau,\lambda)\}\leq 1-\dfrac{\gamma(1-\gamma)}{2(1-\gamma)^2}=\dfrac{2-3\gamma}{2(1-\gamma)}=\alpha'.
		\end{equation*}
		This clearly shows that $\RE\{\xi(i \tau,\lambda)\} \notin \Omega_{\alpha'}$. Now, from~\cite[Theorem 2.3i., p.35]{miler2}, we conclude $\RE p(z) >0$ for $z \in \mathbb{D}$. Therefore we obtain $\mathcal{S}^*(q_\alpha) \subset Q(\gamma(\alpha))$. In addition, the fact $\mathcal{S}_q^*(\alpha) \subset \mathcal{S}^*(q_\alpha)$ proves the result .
	\end{proof}
	\section{Radius Problems}\label{sec5}
	\begin{theorem}
		Let $f \in \mathcal{SL}^*$. Then $f \in \mathcal{S}^*_{q}(\alpha)$ whenever $|z|<\tilde{r}(\alpha)<1$, where $\tilde{r}(\alpha)$ is the smallest positive root of
		\begin{equation}\label{alpha}2(1-r)(\sqrt{1-r}-\alpha)-r=0 \quad (0\leq\alpha<1).\end{equation} The result is sharp.
	\end{theorem}
	\begin{proof}
		For $f \in \mathcal{SL}^*$, there exists a Schwarz function $\upsilon(z)= R e^{i t}$ $(0\leq R\leq r=|z|<1;0\leq t \leq 2\pi)$ such that
		$z f'(z)/f(z) =\sqrt{1+\upsilon(z)}.$  
		A computation shows
		\begin{equation}\label{eq} \min_{|\upsilon(z)|=R}\RE\left(\sqrt{1+\upsilon(z)}\right)= \sqrt{1-R} \geq  \sqrt{1-r}\end{equation} and the following from Schwarz Pick inequality
		\begin{equation}\label{sch}
			\dfrac{|z||\upsilon'(z)|}{1-|\upsilon(z)|} \leq \dfrac{|z|(1+|\upsilon(z)|)}{1-|z|^2} \leq \dfrac{|z|}{1-|z|}.
		\end{equation}
		To prove the result, it suffices to show~\eqref{class} holds.  Therefore we consider
		\begin{align}\label{class1}
			\RE\left(\sqrt{1+\upsilon(z)}\right) - \left|\dfrac{z \upsilon'(z)}{2(1+\upsilon(z))}-\alpha\right| & \geq \RE\left(\sqrt{1+\upsilon(z)}\right)-\alpha-\dfrac{|z||\upsilon'(z)|}{2(1-|\upsilon(z)|)}\nonumber \\ & \geq \sqrt{1-r}-\alpha-\left(\dfrac{r}{2(1-r)}\right), 
		\end{align}
		using the inequalities~\eqref{eq} and~\eqref{sch} with $|z|=r$.
		A calculation reveals~\eqref{class1} further becomes greater than $0$ provided  $r< \tilde{r}(\alpha)$.
		For sharpness, let us consider a function $$\tilde{f}(z) = \dfrac{4 z \exp({2 (\sqrt{1 + z} - 1)})}{(1 +\sqrt{1 + z})^2},$$ belonging to $\mathcal{SL}^*$. 
		Let $\alpha$ be given by~\eqref{alpha}, then equality holds in~\eqref{class} for $\tilde{f}(z)$ at $z=-\tilde{r}(\alpha)$. 
	\end{proof}
	\begin{theorem} 
		Let $f \in \mathcal{S}^*_{l}$. Then $f \in \mathcal{S}^*_{q}(\alpha)$ whenever $|z|<\tilde{r}(\alpha)<1$, where $\tilde{r}(\alpha)$ is the smallest positive root of
		\begin{equation*}
			(1-\log(1+r))(1-r)(1-\log(1+r)-\alpha)-r=0 \quad (0\leq \alpha<1).
		\end{equation*} 
	\end{theorem}
	\begin{proof}
		For $f \in  \mathcal{S}^*_{l}$, we have $z f'(z)/f(z) = 1-\log(1+\upsilon(z))$, where $\upsilon(z)= R e^{i t}$ $(0\leq R\leq r=|z|<1;0\leq t \leq 2\pi)$. To prove the result, it suffices to show
		\begin{equation*}
			\RE(1-\log(1+\upsilon(z))) - \left|\dfrac{-z \upsilon'(z)}{(1+\upsilon(z))(1-\log(1+\upsilon(z)))}-\alpha\right| >0.
		\end{equation*} We also have
		$$ \min_{|\upsilon(z)|=R} \RE(1-\log(1+\upsilon(z))) \text{ and } \min_{|\upsilon(z)|=R} |1-\log(1+\upsilon(z))| = 1-\log(1+R) \geq 1-\log(1+r),$$ see~\cite{banga}. 
		By taking these inequalities and~\eqref{sch} with $|z|=r$, we obtain
		\begin{align*}
			\left|\dfrac{-z \upsilon'(z)}{(1+\upsilon(z))(1-\log(1+\upsilon(z)))}-\alpha\right| &\leq \dfrac{|z||\upsilon'(z)|}{(1-|\upsilon(z)|)|1-\log(1+\upsilon(z))|}+\alpha \nonumber 
			\\ &  \leq \dfrac{r}{(1-r)(1-\log(1+r))}+\alpha,
		\end{align*} which further shows 
		\begin{align*}
			&\RE(1-\log(1+\upsilon(z))) -   \left|\dfrac{-z \upsilon'(z)}{(1+\upsilon(z))(1-\log(1+\upsilon(z)))}-\alpha\right|\nonumber \\&  \geq  \dfrac{(1-r)(1-\log(1+r))^2-r}{(1-r)(1-\log(1+r))}-\alpha >0,
		\end{align*} 
		provided $r<\tilde{r}(\alpha)$.  
	\end{proof}
	\begin{theorem}\label{ethm}
		Let $f \in \mathcal{S}^*_{e}$. We have $f \in \mathcal{S}^*_{q}(\alpha)$ in $|z| < \tilde{r}(\alpha)$, where $\tilde{r}(\alpha)$ is the smallest positive root of 
		\begin{equation}\label{roote}
			\begin{cases}
				4(1-r^2)(e^{-r}-\alpha)-(1+r^2)^2=0, & 0 \leq \alpha < 0.246646\\ 
				e^{-r} -\alpha -r =0, & 0.246646 \leq \alpha <1.
			\end{cases}
		\end{equation}
	\end{theorem}
	\begin{proof}
		For $f \in \mathcal{S}^*_{e}$, we have $z f'(z)/f(z) = e^{\upsilon(z)}$, where $|\upsilon(z)| =R \leq r=|z|.$ A straightforward calculation gives
		\begin{equation*}
			\min_{|\upsilon(z)|=r} \RE(e^{\upsilon(z)}) = e^{-R} \geq e^{-r}.\end{equation*}  The following is the well known inequality~\cite{schwarz} for the derivative of Schwarz function
		\begin{equation}\label{der}|\upsilon'(z)| \leq \begin{cases}
				1, & |z| \leq \sqrt{2}-1 \\ \dfrac{(1+r^2)^2}{4r(1-r^2)},& |z| \geq \sqrt{2}-1. 
		\end{cases}\end{equation} Now we show inequality~\eqref{class} holds, by proving $\RE(e^{\upsilon(z)}) - |z \upsilon'(z)-\alpha| >0$. For this, we consider the following with $|z|=r$ and further applying~\eqref{der}, we get
		\begin{align}\label{ez}
			\RE(e^{\upsilon(z)}) - |z \upsilon'(z)-\alpha| \geq \RE(e^{\upsilon(z)}) -\alpha -|z||\upsilon'(z)| \geq \begin{cases} e^{-r}-\alpha-r, & r\leq \sqrt{2}-1 \\ e^{-r}-\alpha-\dfrac{(1+r^2)^2}{4(1-r^2)}, & r \geq \sqrt{2}-1, \end{cases}\end{align} greater than $0$ provided $r<\tilde{r}(\alpha)$, given in~\eqref{roote}. Note that for $\alpha < 0.246646$, $\tilde{r}(\alpha) \in (\sqrt{2}-1,1]$ and for $\alpha \geq 0.246646$, $\tilde{r}(\alpha) \in (0,\sqrt{2}-1]$. This proves the result.
	\end{proof}
	\begin{theorem}\label{sigthm}
		Let $f \in \mathcal{S}^*_{SG}$. We have $f \in \mathcal{S}^*_q(\alpha)$ in $|z| < \tilde{r}(\alpha)$, where $\tilde{r}(\alpha)$ is the smallest positive root of 
		\begin{equation*}	\begin{cases}
				4(1-r^2)(1+e^{-r})(2-\alpha(1+e^r))-(1+r^2)^2(1+e^r)=0, & 0 \leq \alpha < 0.546407 \\ 	(1+e^{-r})(2-\alpha(1+e^r))-r(1+e^r)=0, & 0.546407 \leq \alpha <1.
	\end{cases}\end{equation*}\end{theorem}
	\begin{proof}
		For $f \in \mathcal{S}^*_{SG}$, we have $z f'(z)/f(z) = 2/(1+e^{-\upsilon(z)})$, for some Schwarz function $\upsilon(z)= R e^{it}$, $(0\leq R\leq r=|z|<1;0\leq t \leq 2\pi)$. To prove the result, it suffices to show~\eqref{class} holds for $f$ in consideration here. 
		Therefore we consider the following 
		\begin{equation}\label{eq4}
			\RE\left(\dfrac{2}{1+e^{-\upsilon(z)}}\right) -\left|\dfrac{z \upsilon'(z)e^{-\upsilon(z)}}{1+e^{-\upsilon(z)}}-\alpha\right|\geq \RE\left(\dfrac{2}{1+e^{-\upsilon(z)}}\right) -\alpha -\dfrac{|z| |\upsilon'(z)|}{|1+e^{\upsilon(z)}|}
		\end{equation}
		From \cite{priyanka}, we have $\min_{|\upsilon(z)|=R} \RE(2/(1+e^{\upsilon(z)})) = 2/(1+e^R) \geq 2/(1+e^r)$ and a computation shows that
		$$ \min_{|\upsilon(z)|=R} |1+e^{\upsilon(z)}| = 1+e^{-R} \geq  1+e^{-r}.$$ Using these inequalities in right side of the inequality~\eqref{eq4} and further applying~\eqref{der} with $|z|=r$,~\eqref{eq4} finally reduces to 
		\begin{align*}
			\RE\left(\dfrac{2}{1+e^{-\upsilon(z)}}\right) -\left|\dfrac{z \upsilon'(z)e^{-\upsilon(z)}}{1+e^{-\upsilon(z)}}-\alpha\right|   \geq  \begin{cases} \dfrac{2}{1+e^{r}} -\alpha-\left(\dfrac{r}{1+e^{-r}}\right),&r \leq \sqrt{2}-1 \\ & \\ \dfrac{2}{1+e^{r}} -\alpha-\left(\dfrac{(1+r^2)^2}{4(1-r^2)(1+e^{-r})}\right),&r \geq \sqrt{2}-1, \end{cases}
		\end{align*}
		greater than $0$ provided $r<\tilde{r}(\alpha)$. Note that for $\alpha <0.546407$, $\tilde{r}(\alpha) \in (\sqrt{2}-1,1]$ and for $\alpha \geq 0.546407$, $\tilde{r}(\alpha) \in (0,\sqrt{2}-1]$.
		This completes the proof. 
	\end{proof}
	\begin{theorem}\label{ssthm}
		Let $f \in \mathcal{S}^*_{S}$.Then $f \in \mathcal{S}^*_{q}(\alpha)$ in $|z|< \tilde{r}(\alpha)$, where $\tilde{r}(\alpha)$ is the smallest positive root of 
		\begin{equation*}
			(1-\sin r)(1-\sin r \cosh r-\alpha)- r \cosh r=0 \quad (0\leq \alpha <1).
		\end{equation*}
	\end{theorem}
	\begin{proof}
		Let $f \in \mathcal{S}^*_{S}$. Then for some Schwarz function $\upsilon(z)= R e^{it}$ $(0\leq R\leq r=|z|<1; 0\leq t\leq 2\pi)$, we have $z f'(z)/f(z) = 1+\sin(\upsilon(z)).$ To prove the result, we show~\eqref{class} holds for $f$ considered here. 
		For this we consider 
		\begin{align}\label{ss1}
			\RE\left(1+\sin(\upsilon(z))\right) -\left|\dfrac{z \upsilon'(z)\cos(\upsilon(z))}{1+\sin(\upsilon(z))}-\alpha\right| & \geq \RE\left(1+\sin(\upsilon(z))\right) -\alpha -\dfrac{|z| |\upsilon'(z)||\cos(\upsilon(z))|}{|1+\sin(\upsilon(z))|}
		\end{align}A calculation shows that $\RE(1+\sin(\upsilon(z)))= 1+\sin(R \cos t)\cosh(R \sin t)$ and $$\sin(R \cos t) \geq -\sin R, \quad  \cosh(R \sin t) \leq \cosh(R).$$ Thus we have $$\RE(1+\sin(\upsilon(z))) \geq 1-\sin R \cosh R \geq 1-\sin r \cosh r.$$ Also 
		$$ \max_{|\upsilon(z)|=R}|\cos(\upsilon(z))| = \cosh R \leq \cosh r \text{ and } \min_{|\upsilon(z)|=R}|1+\sin(\upsilon(z)| = 1-\sin R \geq 1-\sin r.$$ Using these inequalities in right side of the inequality~\eqref{ss1} and further applying~\eqref{der} with $|z|=r$,~\eqref{ss1} finally reduces to
		\begin{align*}
			\RE\left(1+\sin(\upsilon(z))\right) -\left|\dfrac{z \upsilon'(z)\cos(\upsilon(z))}{1+\sin(\upsilon(z))}-\alpha\right| \geq 1-\sin r \cosh r -\alpha - \left(\dfrac{r \cosh r}{1-\sin r}\right)>0,
		\end{align*}
		provided $r < \tilde{r}(\alpha)\leq\sqrt{2}-1$ for every $0 \leq \alpha<1$. 
	\end{proof}
	\begin{theorem}\label{sc}
		Let $f \in \mathcal{S}^*_{C}$ and $0\leq \alpha<3/4$. Then $f \in \mathcal{S}^*_{q}(\alpha)$ in $|z| <\tilde{r}(\alpha)$, where $\tilde{r}(\alpha)$ is the smallest positive root of
		$$(3-2r^2)(1-r^2-2\alpha)-6\sqrt{3}r(1+r)=0 \quad (0\leq \alpha<3/4).$$ 
	\end{theorem}
	\begin{proof}
		For $f \in \mathcal{S}^*_{C}$, there exists a Schwarz function $\upsilon(z) = R e^{it}$ $(0\leq R\leq r=|z|<1;0\leq t\leq 2\pi)$ such that $z f'(z)/f(z)=1+\tfrac{4 \upsilon(z)+2 \upsilon^2(z)}{3}$. We consider
		\begin{align}\label{SC}\RE\left(1+\dfrac{4 \upsilon(z)+2\upsilon^2(z)}{3}\right)-\dfrac{4}{3}\left|\dfrac{z\upsilon'(z)(1+\upsilon(z))}{1+\dfrac{4\upsilon(z)+2\upsilon^2(z)}{3}}-\alpha\right| \geq & \RE\left(1+\dfrac{4 \upsilon(z)+2\upsilon^2(z)}{3}\right)  \nonumber \\ & -\dfrac{4\alpha}{3}-\dfrac{4}{3}\left(\dfrac{|z||\upsilon'(z)|(1+|\upsilon(z)|)}{|1+\dfrac{4\upsilon(z)+2\upsilon^2(z)}{3}|}\right). \end{align}
		Let $\RE\left(1+\tfrac{4 \upsilon(R e^{it})+2\upsilon^2(R e^{it})}{3}\right)=:\kappa(R,t)$, then a calculation shows that $\kappa_t(R,t)=0$ at either $R=0$ or $t=0,~\pi,~2\pi,~\arccos(-1/(2R))$. Evaluating $\kappa(R,t)$ at these values of $R$ and $t$, we get
		$$\min_{0<R\leq r}\kappa(R,t)= \dfrac{2}{3}\left(1-R^2\right) \geq \dfrac{2}{3}\left(1-r^2\right) \text{ and } \kappa(0,t)= 1.$$ Similarly, we calculate
		$$\min_{|\upsilon(z)|=R\neq 0} \left|1+\dfrac{4\upsilon(z)+2\upsilon^2(z)}{3}\right|= \dfrac{3-2R^2}{\sqrt{27}} \geq \dfrac{3-2r^2}{\sqrt{27}}~(r\neq 0) \text{ and } \left|1+\dfrac{4\upsilon(z)+2\upsilon^2(z)}{3}\right|\bigg|_{R=0}=1.$$
		Using these inequalities in right side of the inequality~\eqref{SC} and further applying~\eqref{der} with $|z|=r$,~\eqref{SC} finally reduces to
		$$ \RE\left(1+\dfrac{4 \upsilon(z)+2\upsilon^2(z)}{3}\right)-\dfrac{4}{3}\left|\dfrac{z\upsilon'(z)(1+\upsilon(z))}{1+\dfrac{4\upsilon(z)+2\upsilon^2(z)}{3}}-\alpha\right| \geq \begin{cases}\dfrac{2}{3}(1-r^2)-\dfrac{4\alpha}{3}-\dfrac{4\sqrt{3}r(1+r)}{3-2r^2} & r \neq 0 \\ 1-\dfrac{4\alpha}{3} & r=0,\end{cases}$$ greater than $0$ provided $r<\tilde{r}(\alpha) \leq \sqrt{2}-1$ and $0\leq\alpha<3/4$.
	\end{proof}
	\begin{remark}
		Interestingly in Theorem~\ref{ssthm}--Theorem~\ref{sc}, the respective root $\tilde{r}(\alpha)$ lies in  $(0, \sqrt{2}-1]$.
	\end{remark}
	We carry on with this section by the following radius problems for $\mathcal{S}^*(q_\alpha)$.
	\begin{theorem}\label{rad}
		Let $f \in \mathcal{S}^*(q_\alpha)$ $(0 \leq \alpha <1)$. Then the following holds:
		\begin{itemize}
			\item[(i)] $f$ is starlike of order $\gamma$ in $|z| < \tilde{r}$ whenever \\
			$(a)$ $\tfrac{2 \alpha-1}{2(1-2^{1-2 \alpha})} < \gamma <1$  for $\alpha \neq 1/2$,  where $\tilde{r}$ is the smallest such $r<1$ satisfying the equation 
			\begin{equation*}
				(2\alpha-1)r-\gamma(1+r)(1-(1+r)^{1-2\alpha})=0, \end{equation*} or whenever\\
			$(b)$ $\tfrac{1}{\log 4} < \gamma <1$ for $\alpha =1/2$, 
			where $\tilde{r}$ is the smallest such $r<1$ satisfying the equation 
			\begin{equation*}
				r-\gamma(1+r)\log(1+r)=0.
			\end{equation*} 
			\item[(ii)] Let $\beta >1$ then $f \in \mathcal{M}(\beta)$ in $|z| <r_0$ if there exists $r_0$, the smallest such $r<1$ satisfying the equation 
			\begin{equation*}
				(1-2\alpha)r-\beta(1-r)(1-(1-r)^{1-2\alpha})=0~(\alpha\neq 1/2)\quad \text{ or } \quad 
				r+\beta(1-r)(\log(1-r))=0~(\alpha = 1/2),	
			\end{equation*} else $r_0 =1$.
		\end{itemize}
	\end{theorem}
	\begin{proof}
		As $f \in \mathcal{S}^*({q_\alpha})$, then $f$ satisfies the subordination~\eqref{qa}.  
		
		(i) Let $\alpha \neq 1/2$. Thus we have \begin{equation*}
			\RE\dfrac{z f'(z)}{f(z)} \geq \dfrac{(2\alpha-1)r}{(1+r)(1-(1+r)^{1-2\alpha})}=:\rho(r,\alpha), \quad |z|=r<1.
		\end{equation*}
		We find $r \in [0,1)$ such that $\rho(r,\alpha) > \gamma$ for given $\alpha$ and $\gamma$. Let $\tilde{\rho}(r,\alpha,\gamma) := (2\alpha-1)r/((1+r)(1-(1+r)^{1-2\alpha}))-\gamma$. Now, we observe $\tilde{\rho}(0,\alpha,\gamma)=1-\gamma>0$ and $\tilde{\rho}(1,\alpha,\gamma)= \tfrac{2 \alpha-1}{2(1-2^{1-2 \alpha})}-\gamma < 0$ for every $\gamma$ in the given range. So, there must exist $\tilde{r}$ such that $\tilde{\rho}(r,\alpha,\gamma) \geq 0$ for all $r \in (0,\tilde{r}]$, where $\tilde{r}$ is as described in the hypothesis. Similarly we can prove the result for $\alpha=1/2$. This completes the proof for part (i).\\
		(ii) Let $\alpha\neq1/2$. We have
		\begin{equation*}
			\RE\dfrac{z f'(z)}{f(z)} \leq \dfrac{(1-2\alpha)r}{(1-r)(1-(1-r)^{1-2\alpha})}=:\mu(r,\alpha), \quad |z|=r<1.
		\end{equation*}
		Now, to find $r \in [0,1)$ such that $\mu(r,\alpha)-\beta<0$ for given $\alpha$ and $\beta$, let us assume $$\tilde{\mu}(r,\alpha,\beta):=	\dfrac{(1-2\alpha)r}{(1-r)(1-(1-r)^{1-2\alpha})}-\beta.$$ Clearly, $\tilde{\mu}(0,\alpha,\beta) < 0$ and suppose $\tilde{r}:=1-\epsilon$, $(\epsilon \approx 0)$, then $\mu(r,\alpha) >0$, which implies  either $\tilde{\mu}(\tilde{r},\alpha,\beta) 
		<0$ or $\tilde{\mu}(\tilde{r},\alpha,\beta) >0$ depending on the value of $\beta$. Thus, there must exist $r_0$ such that $\tilde{\mu}(r,\alpha\,\beta) \leq 0$ for all $r\in (0,r_0]$, where $r_0$ is as defined in the hypothesis. On similar lines, proof follows for $\alpha=1/2$, concluding the proof for part (ii).\\ 
	\end{proof}
	\begin{remark}
		In Theorem~\ref{rad}(i), we exclude the case when $0 \leq \gamma \leq \tfrac{2 \alpha-1}{2(1-2^{1-2 \alpha})}$ for $\alpha \neq 1/2$ and $0 \leq \gamma \leq \tfrac{1}{\log 4}$ for $\alpha=1/2$ as for these ranges of $\gamma$, the result holds in $|z| <1$ and it becomes the inclusion result instead, given in Theorem~\ref{inc}.  
	\end{remark}
	\begin{theorem}\label{rad2}
		Let $f \in \mathcal{S}^*(q_\alpha)$ and $\beta \geq 1.$ Then for $|z|<\tilde{r}$, we have
		$$\left|\dfrac{z f'(z)}{f(z)}-1\right| < \beta,$$where  $\gamma:=\gamma(\alpha)$ given by~\eqref{gamma} and $\tilde{r}$ is the smallest such $r<1$ satisfying the equation
		\begin{equation*}
			2(1-\gamma)r-\beta(1-r)(1-|1-2\gamma|r)=0.
		\end{equation*}
	\end{theorem}
	\begin{proof}
		As $f \in \mathcal{S}^*(q_\alpha)$, thus from~\eqref{subor} we deduce the following
		\begin{equation*}
			\dfrac{f(z)}{z} = \dfrac{1+(1-2\gamma)\upsilon(z)}{1-\upsilon(z)},\end{equation*} for some Schwarz function $\upsilon(z)$. Logarithmic differentiating the above equation yields
		\begin{equation*}\dfrac{z f'(z)}{f(z)}-1=\dfrac{2(1-\gamma)z \upsilon'(z)}{(1+(1-2\gamma)\upsilon(z))(1-\upsilon(z))}.\end{equation*}
		Using the triangle inequality on the modulus of above equation, we get
		\begin{equation}\label{zf}
			\left|\dfrac{z f'(z)}{f(z)}-1\right|\leq\dfrac{2(1-\gamma)|z| |\upsilon'(z)|}{(1-|1-2\gamma||\upsilon(z)|)(1-|\upsilon(z)|)}
		\end{equation}
		Upon applying the Schwarz-Pick inequality  $$|\upsilon'(z)| \leq \dfrac{1-|\upsilon(z)|^2}{1-|z|^2} \quad (z \in \mathbb{D})$$ and $|\upsilon(z)|\leq |z|$ in equation~\eqref{zf}, we deduce
		\begin{align*}\left|\dfrac{zf'(z)}{f(z)}-1\right|& \leq \dfrac{2(1-\gamma)|z| (1+|z|)}{(1-|1-2\gamma||z|)(1-|z|^2)}\\ & = \dfrac{2(1-\gamma)r }{(1-|1-2\gamma|r)(1-r)}=:u(r,\gamma).\end{align*}
		Now to show $u(r,\gamma)-\beta <0$. For this consider, $\tilde{u}(r,\gamma,\beta):=2(1-\gamma)r-\beta(1-r)(1-|1-2\gamma|r)$. We observe $\tilde{u}(0,\gamma,\beta)=-\beta<0$ and $\tilde{u}(1,\gamma,\beta)=2(1-\gamma)>0$. Thus, there must exist $\tilde{r}$ such that $\tilde{u}(r,\gamma,\beta)\leq 0$ for all $r \in [0,\tilde{r}]$, where $\tilde{r}$ is defined in the hypothesis. Hence the result.		
	\end{proof}
	\begin{remark}
		Since $\mathcal{S}_q^*(\alpha) \subset \mathcal{S}^*(q_\alpha)$, results in Theorem~\ref{rad}--Theorem~\ref{rad2} hold even for $\mathcal{S}_q^*(\alpha)$, however, they need not be sharp.
	\end{remark}
	\subsection*{Acknowledgment}
	The second author is supported by a Research Fellowship
	from the Department of Science and Technology, New Delhi, Grant No.~IF170272.

\end{document}